\documentclass[a4,12pt]{amsart}
\usepackage{amssymb,amstext,amsmath,amscd,amsthm,amsfonts,enumerate,graphicx,latexsym}
\usepackage[usenames]{color}
\newtheorem{thm}{Theorem}%[section]
\newtheorem{cor}[thm]{Corollary}

\theoremstyle{definition}

\newtheorem{rem}[thm]{Remark}
\newtheorem{ques}[thm]{Question}

\newtheorem*{claim*}{Claim}
\numberwithin{equation}{thm}
\def\Ext{\mathsf{Ext}}
\def\syz{\mathsf{\Omega}}
\def\tr{\mathsf{Tr}}
\def\cm{\mathsf{CM}}
\def\lcm{\mathsf{\underline{CM}}}
\def\db{\mathsf{D^b}}
\def\ds{\mathsf{D_{sg}}}
\def\lm{\operatorname{\mathsf{\underline{mod}}}}
\def\Hom{\mathsf{Hom}}
\def\id{\mathsf{I}}
\def\mod{\operatorname{\mathsf{mod}}}
\def\lhom{\mathsf{\underline{Hom}}}
\begin{document}
\allowdisplaybreaks
\setlength{\baselineskip}{15pt}
\title[Endofunctors of singularity categories]{Endofunctors of singularity categories characterizing Gorenstein rings}
%\date{\today}
\author{Takuma Aihara}
\address{Graduate School of Mathematics, Nagoya University, Furocho, Chikusaku, Nagoya, Aichi 464-8602, Japan}
\email{aihara.takuma@math.nagoya-u.ac.jp}
\author{Ryo Takahashi}
\address{Graduate School of Mathematics, Nagoya University, Furocho, Chikusaku, Nagoya, Aichi 464-8602, Japan}
\email{takahashi@math.nagoya-u.ac.jp}
\urladdr{http://www.math.nagoya-u.ac.jp/~takahashi/}
\thanks{The second author was partially supported by JSPS Grant-in-Aid for Scientific Research (C) 25400038}
%\dedicatory{Dedicated to Professor Craig Huneke on the occasion of his sixtieth birthday}
\begin{abstract}
In this paper, we prove that certain contravariant endofunctors of singularity categories characterize Gorenstein rings.
\end{abstract}
\maketitle
%\tableofcontents
%%%%%%%%%%%%%%%%%%%%%%%%%%%%%%%%%%%%%%%%%%%%%%%%%%%%%%%%

Let $\Lambda$ be a noetherian ring.
Denote by $\ds(\Lambda)$ the {\em singularity category} of $\Lambda$, that is, the Verdier quotient of the bounded derived category $\db(\Lambda)$ of finitely generated (right) $\Lambda$-modules by the full subcategory consisting of bounded complexes of finitely generated projective $\Lambda$-modules.
We are interested in the following question.

\begin{ques}\label{1}
What contravariant endofunctor of $\ds(\Lambda)$ characterizes the Iwanaga--Gorenstein property of $\Lambda$?
\end{ques}

In this paper we shall consider this question in the case where $\Lambda$ is commutative and Cohen--Macaulay.

Let $R$ be a commutative Cohen--Macaulay local ring of Krull dimension $d$.
Denote by $\cm(R)$ the category of (maximal) Cohen--Macaulay $R$-modules and by $\lcm(R)$ its {\em stable category}\,: the objects of $\lcm(R)$ are the Cohen--Macaulay $R$-modules, and the hom-set $\Hom_{\lcm(R)}(M,N)$ is defined as $\lhom_R(M,N)$, the quotient module of $\Hom_R(M,N)$ by the submodule consisting homomorphisms factoring through finitely generated projective (or equivalently, free) $R$-modules.
The natural full embedding functor $\cm(R)\to\db(R)$ induces an additive covariant functor
$$
\eta:\lcm(R)\to\ds(R).
$$
Furthermore, the assignment $M\mapsto\syz^d\tr M$, where $\syz$ and $\tr$ stand for the syzygy and transpose functors respectively (see \cite[Chapter 2, \S1]{AB} for details of the functors $\syz$ and $\tr$), makes an additive contravariant functor
$$
\lambda:\lcm(R)\to\lcm(R).
$$

The following result gives a partial answer to Question \ref{1}.

\begin{thm}
The following are equivalent.
\begin{enumerate}[\rm(1)]
\item
The ring $R$ is Gorenstein.
\item
The functor $\eta$ is an equivalence (i.e. $\eta$ is full, faithful and dense).
\item
There exists a functor $\phi:\ds(R)\to\ds(R)$ such that the diagram
$$
\begin{CD}
\ds(R) @>\phi>> \ds(R) \\
@A{\eta}AA @AA{\eta}A \\
\lcm(R) @>>\lambda> \lcm(R)
\end{CD}
$$
of functors commutes up to isomorphism.
\end{enumerate}
\end{thm}

\begin{proof}
(1) $\Rightarrow$ (2):
If $R$ is Gorenstein, then a celebrated theorem of Buchweitz \cite[Theorem 4.4.1]{B} implies that the functor $\eta$ is an equivalence.

(2) $\Rightarrow$ (3):
When $\eta$ is an equivalence, we have a contravariant endofunctor
$$
\phi=\eta\lambda\rho:\ds(R)\to\ds(R).
$$
of $\ds(R)$, where $\rho$ stands for a quasi-inverse of $\eta$.
Condition (3) holds for this functor $\phi$.

(3) $\Rightarrow$ (1):
In the remainder of the proof, we will omit writing free summands.
Let
$$
\pi:\mod R\to\lm R
$$
be the canonical functor from the category of finitely generated $R$-modules to its stable category, that is, the objects of $\lm R$ are the finitely generated $R$-modules and the hom-set $\Hom_{\lm R}(M,N)$ is defined as $\lhom_R(M,N)$.

Assume that there are a contravariant functor $\phi:\ds(R)\to\ds(R)$ and an isomorphism
$$
\Delta:\phi\eta\to\eta\lambda
$$
of functors from $\lcm(R)$ to $\ds(R)$.
Take a Cohen--Macaulay $R$-module $M$.
It follows from \cite[Proposition (2.21)]{AB} that there exists an exact sequence
\begin{equation}\label{2}
0 \to F \to \tr\syz\tr\syz M \xrightarrow{f} M \to 0
\end{equation}
of finitely generated $R$-modules with $F$ free.
The map $f$ induces a morphism
$$
\Theta:\tr\syz\tr\syz\to\id
$$
of functors from $\lcm(R)$ to $\lcm(R)$, where $\id$ stands for the identity functor.
Applying the $R$-dual functor $(-)^*=\Hom_R(-,R)$ to \eqref{2} gives an exact sequence
$$
0 \to M^* \xrightarrow{f^*} (\tr\syz\tr\syz M)^* \xrightarrow{g} F^* \to \tr M \xrightarrow{h} \tr(\tr\syz\tr\syz M) \to \tr F \to 0
$$
with $\pi(h)=\tr\pi(f)$; see \cite[Lemma (3.9)]{AB}.
Note that there is also an exact sequence
$$
0 \to M^* \xrightarrow{f^*} (\tr\syz\tr\syz M)^* \xrightarrow{g} F^* \to \Ext_R^1(M,R) \to \Ext_R^1(\tr\syz\tr\syz M,R).
$$
Since $\Ext_R^1(\tr\syz\tr\syz M,R)=0$ by \cite[Theorem (2.17)]{AB} and since $\tr F$ is free, we obtain an exact sequence
$$
0 \to \Ext_R^1(M,R) \to \tr M \xrightarrow{h'} \tr(\tr\syz\tr\syz M) \to 0
$$
such that $\pi(h')=\pi(h)$.
Taking the $d$-th syzygies of $\Ext_R^1(M,R)$ and $\tr(\tr\syz\tr\syz M)$ and using the horseshoe lemma, we get an exact sequence of Cohen--Macaulay $R$-modules
\begin{equation}\label{3}
0 \to \syz^d\Ext_R^1(M,R) \to \lambda M \xrightarrow{\ell} \lambda(\tr\syz\tr\syz M) \to 0
\end{equation}
with $\pi(\ell)=\lambda\pi(f)$.
Note that for each short exact sequence $\sigma:0\to X\xrightarrow{\alpha}Y\xrightarrow{\beta}Z\to0$ of Cohen--Macaulay $R$-modules, the image of $\sigma$ by the canonical functor $\pi$ is sent by $\eta$ to an exact triangle $X\xrightarrow{\alpha}Y\xrightarrow{\beta}Z\rightsquigarrow$ in $\ds(R)$.
Hence, $\eta$ sends \eqref{3} to an exact triangle
$$
\eta\syz^d\Ext_R^1(M,R) \to \eta\lambda M \xrightarrow{\eta\lambda(\Theta M)} \eta\lambda(\tr\syz\tr\syz M) \rightsquigarrow
$$
in $\ds(R)$.
We have a commutative diagram
$$
\begin{CD}
\eta\lambda M @>{\eta\lambda(\Theta M)}>> \eta\lambda(\tr\syz\tr\syz M) \\
@A{\Delta M}A{\cong}A @A{\cong}A{\Delta(\tr\syz\tr\syz M)}A \\
\phi\eta M @>>{\phi\eta(\Theta M)}> \phi\eta(\tr\syz\tr\syz M)
\end{CD}
$$
of morphisms in $\ds(R)$, and the exact sequence \eqref{2} induces an isomorphism
$$
\eta(\Theta M):\eta(\tr\syz\tr\syz M) \to \eta M
$$
in $\ds(R)$.
Therefore $\eta\syz^d\Ext_R^1(M,R)$ is isomorphic to $0$ in $\ds(R)$, which means that the $R$-module $\Ext_R^1(M,R)$ has finite projective dimension.
Thus, letting $M:=\syz^dk$, where $k$ denotes the residue field of $R$, shows that $\Ext_R^{d+1}(k,R)$ has finite projective dimension.
If $\Ext_R^{d+1}(k,R)=0$, then $R$ is Gorenstein.
If $\Ext_R^{d+1}(k,R)\ne0$, then the $R$-module $k$ has finite projective dimension, which implies that $R$ is regular, so that $\Ext_R^{d+1}(k,R)=0$, a contradiction.
Consequently, in either case $R$ is a Gorenstein ring.

Now the proof of the theorem is completed.
\end{proof}

\begin{rem}
In the proof of the theorem, the assumption that the ring $R$ is commutative is used to deduce the Gorensteinness of $R$ from the fact that the $R$-module $\Ext_R^{d+1}(k,R)$ has finite projective dimension.
For a noncommutative ring $\Lambda$ with Jacobson radical $J$ and an integer $n$, the $n$-th Ext group $\Ext_\Lambda^n(\Lambda/J,\Lambda)$ of the right $\Lambda$-modules $\Lambda/J$ and $\Lambda$ is not necessarily semisimple as a left $\Lambda$-module.
\end{rem}

We end this paper by stating a direct consequence of the theorem.

\begin{cor}
Suppose that $R$ is artinian.
Then $R$ is Gorenstein if and only if the transpose functor $\tr:\lm R\to\lm R$ extends to the singularity category $\ds(R)$.
\end{cor}

%%%%%%%%%%%%%%%%%%%%%%%%%%%%%%%%%%%%%%%%%%%%%%%%%%%%%%%%

%%%%%%%%%%%%%%%%%%%%%%%%%%%%%%%%%%%%%%%%%%%%%%%%%%%%%%%%
\end{document}